\documentclass[preprint,12pt]{elsarticle}
\makeatletter
\def\ps@pprintTitle{%
 \let\@oddhead\@empty
 \let\@evenhead\@empty
 \def\@oddfoot{}%
 \let\@evenfoot\@oddfoot}
\makeatother

 \usepackage{graphicx}

\usepackage{amssymb}
 \usepackage{amsthm}
\newtheorem{theorem}{Theorem}
\newtheorem{lemma}{Lemma}
\newtheorem{definition}{Definition}
\newtheorem{proposition}{Proposition}
\newtheorem{remark}{Remark}

\usepackage{amsmath}

\begin{document}

\begin{frontmatter}



\title{Discretization of Time-Dependent Quantum Systems: 
Real-Time Propagation of The Evolution
Operator}


\author[JJ]{Joseph W. Jerome}
\address[JJ]{Department of Mathematics, Northwestern University,
                Evanston, IL 60208. \\ 
email: jwj@math.northwestern.edu; corresponding author}

\author[EP]{Eric Polizzi}
\address[EP]{Department of Electrical and Computer Engineering,
University of Massachusetts,
               Amherst, MA 01003, \\ 
email: polizzi@ecs.umass.edu.}

\begin{abstract}
We discuss 
time dependent quantum systems on bounded domains. Our work may be viewed as a
framework for several models, including linear iterations involved in 
time dependent density functional theory (TDDFT), the Hartree-Fock model, 
or other quantum models. 
A key aspect of the analysis of the algorithms is the use of 
time-ordered evolution
operators, which allow for both a well-posed problem and its
approximation. The approximation theorems obtained 
for the time-ordered evolution operators 
complement those in the current literature. 
We discuss the available theory at the outset, and proceed to
apply the theory systematically in later sections via approximations
and a global existence theorem for a nonlinear system, 
obtained via a fixed point theorem for the evolution operator.   
Our work is consistent with first-principle real time propagation of 
electronic states, aimed at finding 
the electronic responses of quantum molecular systems and nanostructures. 
We present two full 3D quantum atomistic simulations 
using the finite element method for discretizing the real-space, and the FEAST eigenvalue algorithm 
for solving the evolution operator at each time step. These numerical experiments
are representative of the theoretical results.
\end{abstract}

\end{frontmatter}

\section{Introduction} 
This article analyzes a general version of time dependent 
quantum mechanical systems via time ordered evolution operators. 
Time-ordered evolution operators arise from direct integration of 
the time-dependent 
Schr\"odinger equation. 
They are most often used to enable real-time propagation of 
ground-state solutions 
in response to any arbitrary external perturbations of the quantum system.
Important physics can be extracted from the time-domain responses. 
The development of efficient numerical techniques which aim at achieving 
both accuracy and performance 
in time-dependent quantum simulations has become important for a 
large number of applications spanning the fields of quantum chemistry, 
solid state physics and spectroscopy. In particular, 
finding a suitable numerical representation for the time-ordered evolution
operator is one of the main focuses of the TDDFT research field \cite{tddft}.

The numerical treatment of time-ordered evolution operators often gives rise to
the matrix exponential, commonly treated using approximations such as 
split-operator techniques \cite{MP}. 
The efficiency of the time-domain propagation techniques described here, 
however, is further enhanced by reliance 
on the capabilities of the new FEAST algorithm for solving the 
eigenvalue problem \cite{POL,feast}. By using  
FEAST, the solution of the eigenvalue problem  is reformulated into 
solving a set of well-defined independent
linear systems along a complex energy contour.  
Obtaining the spectral decomposition of the matrix exponential 
becomes then a suitable alternative to PDE based techniques such as 
Crank-Nicolson schemes \cite{CN}, and can also take advantage of parallelism.

The goals of the paper are as follows.
\begin{itemize}
\item
To provide a rigorous infrastructure, both on 
the ground space and the `smooth' space, 
for the evolution operator
used in topical applications of TDDFT cited in this article (see \cite{Y}
for an early adaptation of Kato's evolution operator);
\item To complement the numerical Gauss quadrature in time introduced in \cite{CP} and 
to provide an exact interface with the use of FEAST; the simulations and
theorems of this article are tightly connected; 
\item To complement
the detailed estimates obtained via the Magnus expansion
\cite{Mag,AF} by an alternative approach based on finite element estimation; in
particular, the Bramble-Hilbert lemma and the Sobolev representation
theorem;
\item
To introduce the numerical evolution operators; in an approximate sense,
this leads to the approximate preservation of significant quantities. 
\item
To obtain, via entirely different methods based on the evolution
operator, pertinent 
existence theorems in the
literature \cite{CLe,Caz}; in some cases, more information can be extracted from
this approach, including local existence for very general nonlinearities. 
In particular, our focus on the nonlinear Schr\"{o}dinger
equation with Hartree potential is consistent with recent studies
\cite{EESY} characterizing this equation as a weak limit of weakly coupled
Fermion systems. Our global analytical methods are not applied to obtain
uniqueness for nonlinear systems, since this is a well-studied topic.
\end{itemize} 
We summarize now the plan of the paper. In the following section, we
outline the mathematical properties developed over the years for
Schr\"{o}dinger operators, as applied to many-particle systems. 
The section includes a discussion of current understanding and practice. 
We also introduce the evolution operator and admissible Hamiltonians.
In the appendix, we include the basic theory of the evolution operator.
This is due to Kato \cite{K1,K2} and Dorroh \cite{D}, and is detailed in
\cite{J1}. The appendix includes the verification that the Hartree
potential satisfies the required hypotheses for inclusion in the class of
admissible Hamiltonians; this leads to invariance of the evolution
operator on the smooth space.
Section three introduces discretization of the evolution
operators, in terms of the
traditional rectangular rule, for short time steps, and in terms of  
`degree of precision' quadrature rules for longer time steps.
Although this resonates with classical theory, the corresponding proofs
of the approximation theorems of the following section are not elementary.
This is followed in section four by precise statements of the principal
theorems and by proofs, which validate the discretizations. 
In addition, a well-posedness result (global existence in time) is given
for the nonlinear Schr\"{o}dinger equation, involving the Hartree
potential coupled to an external potential for a closed system.
Numerical simulations and discussions are presented in section five, 
and future research
is outlined in section six. Finally, our analysis is for the bounded
domain in Euclidean three space, and excludes the use of Strichartz
estimates.  

\section{Time Dependent Quantum Systems} 
Two major theories have been developed to analyze many-particle 
quantum systems.
Classical density functional theory (DFT) is derived
from the Hohenberg-Kohn theorem in \cite{KV}.
By transferring inter-electron effects to the exchange-correlation
potential, expressed as a functional of the electron density $\rho$, 
the theory is capable of representing 
a many-electron system in terms of non-interacting
effective particles. This theory employs pseudo-wave functions 
but a precise representation for the electron charge density.
The aggregate potential is the effective potential 
$V_{\rm eff}$. 
This leads to the Hamiltonian $\hat H$ 
and its associated Kohn-Sham orbitals \cite{KS}.
Well-posedness of the steady problem
has been studied in \cite{PN};
applications in \cite{FW}.
Although DFT is only applicable for obtaining the ground state of 
quantum systems consistent with charge density,  
its time-dependent counterpart, TDDFT, has been proposed in  \cite{RG}
to investigate the dynamics of many-body systems and can be potentially used to obtain energies
of excited states. 
Another major theory
used in the theoretical chemistry community is the
Hartree-Fock model. Here, the emphasis is directed toward the calculation
of exact orbitals (for a mathematical discussion, cf. \cite{CL}).
Aspects of these two theories are covered in the present framework, as
well as other quantum theories.
\subsection{Initial value problem for Schr\"{o}dinger systems}
We follow the notation and format of \cite{CP}. If we denote by $\hat H$
the Hamiltonian operator of the system, then the state $\Psi(t)$ of the
closed quantum system obeys the Schr\"{o}dinger equation,
\begin{equation}
\label{eeq}
i \hbar \frac{\partial \Psi(t)}{\partial t} = \hat H \Psi(t).
\end{equation}
For mathematical well-posedness, an initial condition,
\begin{equation}
\label{ic}
\Psi(0) = \Psi_{0},
\end{equation}
and boundary conditions must be adjoined. 
In the study \cite{RG}, it is shown that the initial value problem is
well-defined physically; there is an inherent invertible mapping 
from the time
dependent external potential function to the time dependent particle density.  
This study is now the basis in the physics community for the reliability
of physics-based studies involving time dependent density functional
theory. An important study using this model is contained in \cite{CMR}.
The model is now characterized as the Runge-Gross model; the potential
in the Hamiltonian
includes: an external potential, which allows for an
ionic component, 
the Hartree potential,
and the exchange-correlation potential.
Except for the global existence result for the Hartree potential, 
the present article is restricted to potentials which are linear in the
quantum state, but the theory has the capacity to extend to 
local (in time) nonlinear versions of the Runge-Gross model. 
We will assume that the particles are confined to a
bounded region $\Omega \subset {\bf {R}}^{d}$, with $d = 1,2,3$, and that
homogeneous Dirichlet boundary conditions hold for the evolving quantum
state. In particular, the spectrum of the Hamiltonian is discrete in this
case. Also, the proofs are unaffected by the interpretation of $\Psi$
as a scalar or vector complex-valued function.
\subsection{Specification of the Hamiltonian operator}
\label{Ham}
Consider a linear problem, i.\thinspace e.\thinspace,
an external potential $V({\bf x}, t)$ which is independent of the system state,
particularly the charge density. This assumption 
is equivalent to studying a non-interacting system. Alternatively, in the
case of an interacting system, it describes 
exactly one iteration of a nonlinear mapping based on potential which
includes the exchange correlation potential, the Hartree potential,
or contributions from other
quantum system models.  
It is natural
therefore to make the following assumption:

{\bf Assumption} The Hamiltonian, 
\begin{equation}
-\frac{\hbar^{2}}{2m} \nabla^{2} + V(\cdotp, t),
\label{eq_pot}
\end{equation}
has, for each $t$, an $L^{2}$ self-adjoint extension $\hat H(t)$.

It follows from a theorem of Stone \cite{S}, \cite[Ch.\ 35, Theorem
1]{Lax} that, for each fixed $t_{\ast}$,  
$(\pm i/\hbar)\hat H(t_{\ast})$ is the infinitesimal generator of a strongly
continuous group, 
$\{\exp[(\pm i/\hbar)\hat H(t_{\ast})t]\}$, 
of unitary operators on $L^{2}$.

The earliest results for the self-adjointness of the Hamiltonian with
interactions including Coulomb potentials are attributed to Kato
\cite{K3,K4}.
Since later results by Kato and other authors \cite{CS} 
imply that these operators
are also stable in the sense we have defined them, it follows that the
framework for evolution operators outlined here covers this case.
Moreover, any further perturbation of such potentials by potentials  
depending (non-linearly) on $C^{1}$ class functions of the quantum state, 
with bounded derivatives, 
is also admissible. This is a
classical commutator result, initially investigated in \cite{Cal}. 
The framework here is thus quite general. 
However, the choice of $Y$ of Theorem \ref{UE} of the Appendix, is strongly dependent on the structure of the
effective potential.
\subsection{The Hartree potential and admissible external potentials}
\label{admissible}
In this section, all statements pertain to Euclidean space ${\bf R}^{3}$.
In order to motivate the format of the Hamiltonian operators for the
linear problem, we first consider the
initial value problem for the {\it nonlinear} Schr\"{o}dinger
equation,
\begin{equation}
\label{eeq2}
i \hbar \frac{\partial \Psi(t)}{\partial t} = \hat H \Psi(t),
\end{equation}
where $$\hat H \Psi =-\frac{\hbar^{2}}{2 m}\nabla^{2}\Psi + V_{\rm ex}\Psi 
+ (W \ast |\Psi|^{2}) \Psi. $$ 
Here, $W({\bf x}) = 1/|{\bf x}|$, and the convolution, 
$$(W \ast |\Psi|^{2})({\bf x}, t) = 
\int_{\Omega}W({\bf x} - {\bf y}) |\Psi({\bf y}, t)|^{2}
\; d{y_{1} dy_{2}dy_{3}},$$ 
represents the Hartree potential, where we
have written 
$|\Psi|^{2} $ for the charge density $\rho$, and $V_{\rm ex} = 
V_{\rm ex}({\bf x},t)$
for the external potential. When spin is accounted for, $\rho$ includes an
additional factor of two. 
In the appendix, we are able to show that, for a choice of Hartree potential
defined by a charge density of minimal regularity, the Hamiltonian family
may be used to construct the evolution operators $\{\hat {U}(t,s)\}$. 

\section{Discretization Schemes}
We begin by introducing a widely used notation in the 
mathematical physics community (e.\thinspace g.\thinspace, see \cite{CMR})
for the evolution operators $\{\hat {U}(t,s)\}$, 
which can be useful if the argument $(t,s)$
is not essential, and emphasis is to be placed upon the family of 
semigroup generators and the semigroups used in the construction of the
evolution operators. 
Formally, then,
the time-ordered evolution operator for (\ref{eeq}) takes the form
\cite{CMR}:
\begin{equation}\label{oper1}
\hat{U}(t,0)=\mathcal{T}\exp\left\{  -\frac{i}{\hbar}\int^t_0d\tau 
\hat{H}(\tau)  \right\},
\end{equation}
and the final solution at time $T$ is then given by: 
\begin{equation}\label{prop}
\Psi(T)=\hat{U}(T,0)\Psi_0.
\end{equation}
This is equivalent to the formula (\ref{SCP}) with $F = 0$.
Notice that $\hat{U}$ is used here for the quantum mechanical interpretation 
of evolution operators.

In addition to the final solution $\Psi(T)$, the evolution of the system 
along $[0,T]$ can be described by intermediate solutions.
From the properties of the time-ordered evolution operator 
(property {\bf II} of Theorem \ref{UE} of the Appendix), one can indeed apply the following decomposition:
\begin{equation}\label{unit}
\hat{U}(T,0)= 
\hat{U}(t_n,t_{n-1})\dots\hat{U}(t_2,t_1)\hat{U}(t_1,t_0),
\end{equation}
where we consider $n-1$ intermediate times with 
$t_0=0$ and $t_n=T$, and where the solution $\Psi(t)$ can be obtained 
at time $t_j$,
$j=1\dots n$. Let us assume a constant time step $\Delta$; 
the corresponding time-ordered evolution operator is designated
\begin{equation}\label{prop1}
\hat{U}(t+\Delta,t)=
\mathcal{T} \exp\left\{-\frac{i}{\hbar}\int^{t+\Delta}_t d\tau \hat{H}(\tau)  
\right\}.
\end{equation}

Let us then outline two possibilities: (i) $\Delta$ is very small in 
comparison to the variation of the potential $V(\cdotp,t)$; and 
(ii) $\Delta$ is much larger.

\subsection{Small time-step intervals: the rectangular rule}

If $\Delta$ is chosen very small such that $\hat{H}(\tau)$ can be considered 
constant within the time interval $[t,t+\Delta]$, it follows that
the argument of the exponential in (\ref{prop1}) needs to be evaluated only 
at time $t$:
\begin{equation}\label{bf1}
\hat{U}_\Delta(t+\Delta,t) \mapsto
\exp\left\{ -\frac{i}{\hbar}\Delta \hat{H}(t)   
\right\},
\end{equation}
which is then equivalent to solving a time independent problem along $\Delta$.
Note that this is equivalent to Definition \ref{recrule}, formulated in
section \ref{A}.
Additionally, we  note that the time-ordered exponential can be replaced 
by the exponential (semigroup)
operator in this case, which is the essence of the
rectangular integration rule. Schematically, we write for the semigroup
product: 
\begin{equation}\label{to}
\mathcal{T}\left\{\prod_j S(t_j)\right\}=S(t_N)\dots S(t_2)S(t_1).
\end{equation}
In section \ref{A}, we show how the rectangular rule globally defines a
family of approximate evolution operators, shown (rigorously) to converge
to the time-ordered family. In this case, the approximation operators must
be defined so that they also possess the time-ordered property.
We have then the following on $[0,T]$: 
\begin{equation}\label{lim1}
\lim_{\Delta\to 0}\hat{U}_\Delta = \hat{U}.
\end{equation}
The approximation order is shown to be $o(\Delta)$ in Theorem
\ref{theoremA}.

\subsection{Long time-step intervals}
In simulations, the use of very small time-step intervals has a 
sound physical interpretation as it corresponds to a step by step 
propagation of the solution over time. The major drawback of this approach, 
however, is that it involves a very large number of time 
steps from initial to final simulation times. 
In contrast, much larger time intervals could become advantageous in 
simulations since the electron density (or other integrated physical quantities)
is likely to exhibit much weaker variations as compared to the variations 
of the individual wave functions.
In addition, at certain frequency, e.g. THz, long-time domain response 
is needed, and accurate calculations using large time steps could be used to 
speed-up the simulation times.
Let us now consider the case of a much longer time interval of length 
$\Delta$,
which may correspond, 
for instance, 
to a given period of a time-dependent perturbation potential 
$V(\cdotp,t)=V_0(\cdotp) \sin(2\pi t/\Delta)$. 
A direct numerical integration of the integral component in the 
time-ordered evolution operator (\ref{prop1}), leads to
\begin{equation}\label{oper111}
\hat{U}_{\delta}(t+\Delta,t)=\mathcal{T}
\exp \left\{  -\frac{i}{\hbar}\xi\sum^p_{j=1}\omega_j \hat {H}(t_j) \right\},
\end{equation}
where $\omega_j$ and $\xi$ are integration weights, 
and $p$ is the number of quadrature points.  
The subscript $\delta$ suggests the local construction of the evolution
operators within the larger subinterval.
\begin{remark}
In the case of a rectangular quadrature rule,
one notes that $\omega_j=1$, $\xi=\delta$, $t_j=t+j*\delta$ and 
$\delta\equiv \Delta/(p+1)$. Here, $j = 0, \dots, p+1$. 
Therefore, it follows from (\ref{lim1}):
$$\lim_{p\to \infty}\hat{U}_\delta = \hat{U}.$$
In particular, if the number of rectangle 
quadrature points $p$ increases significantly,
the problem is then equivalent to solving (\ref{bf1}) multiple times since
\begin{equation}\label{unit2}
\hat{U}_\delta(t+\Delta,t)=\prod_{j=0}^p  \hat{U}_\delta(t_j+\delta,t_j).
\end{equation}
\end{remark} 

Clearly, higher-order quadrature schemes such as Gaussian quadrature 
 can use far fewer points $p$ than a
low-order quadrature rule such as the rectangular rule, 
to yield a high order approximation of the integral of a function.
A $p$-point Gaussian quadrature rule is a numerical integration constructed 
to yield
an exact result for polynomials of degree $2p-1$ by a suitable
choice of the points $t_i$ and Gauss-Legendre weights $\omega_j$ 
\cite[Sec. 5.5]{Sauer}.
We associate the quadrature points $t_j$ at the Gauss node $x_j$
using $t_j=\frac{\Delta}{2}x_j+\frac{2t_{0}+\Delta}{2}$; 
also we note $\xi=\Delta/2$. Thus, the following is a reasonable
conjecture:

$$\forall \epsilon, \quad \exists p_0 \mbox{ such that } 
\forall p\geq p_0, \quad \|\hat{U}_{\delta}-\hat{U}\|\leq\epsilon.$$
Here, $\delta$ represents an average spacing between 
quadrature nodes: $\delta\simeq \Delta/(p+1)$. 
In section \ref{B}, we show that this estimate is rigorously correct for the
weighted sum of the Hamiltonians (cf. Theorem \ref{theoremB}).

\subsection{Evaluation of the approximate evolution operator}
In  order to evaluate the time-ordered evolution operator, 
it is necessary to decompose the exponential in (\ref{oper111})
into a product of exponential operators taken at different time steps: 
 \begin{equation}\label{oper1111}
\hat{U}_{\delta}(t+\Delta,t)=
\mathcal{T}\left\{\prod_{j=1}^{p}\exp\left\{-\frac{i}{\hbar}\xi\omega_j 
\hat {H}(t_j)\right\} \right\}+O[\delta],
\end{equation}
which expresses an anti-commutation error $O[\delta]$ between 
Hamiltonian operators evaluated at different times $t_j$.
The validity of this approximation is discussed in section \ref{C}.

We note from equations (\ref{oper111}) and (\ref{oper1111}) that
 two numerical errors are respectively involved:
 (i) a quadrature error resulting from the discretization of the integral 
and (ii) an anti-commutation error resulting from 
 the decomposition of the exponential operators.

\section{Principal Theorems}
This section is devoted to theorems \ref{theoremA}, \ref{theoremB}, and
\ref{global.existence}.
\subsection{Convergence of the rectangular approximation}
\label{A}
We present a general result, not restricted to the quantum application.
\begin{definition}
\label{recrule}
Given $\{A(t)\}$ as in Definition \ref{def3.1}, define
$$
A_{n}(t) = A(T[nt/T]/n), \; 0 \leq t \leq T.
$$
Here, $[s]$ denotes the greatest integer less than or equal to $s$.
If $s \leq t$, and $s,t \in [t_{j-1}, t_{j}]$, and $A_{n} \equiv A$ on
this interval, then
$$
U_{n}(t,s) = e^{-(t-s)A}.
$$
For other values of $s,t$, $U_{n}(t,s)$ is uniquely determined by the
condition 
$$
U_{n}(t,r) = U_{n}(t,s)U_{n}(s,r).
$$
\end{definition} 

We make the following observations. 
\begin{itemize}
\item
Convergence of generator approximations as $n \rightarrow \infty$:
$$
\|A(t) - A_{n}(t)\|_{Y,X} \rightarrow 0, \; \mbox{uniformly}, t \in [0,T].
$$
\item
Invariance and uniform boundedness of evolution operators on $Y$:
$$
U_{n}(t,s) Y \subset Y, \; \|U_{n}(t,s)\|_{Y} \leq C(T), \; \forall t,s,n.
$$
\item
Differentiation:
$$
(d/dt) U_{n}(t,s)g = -A_{n}(t)U_{n}(t,s) g, \; g \in Y, \; \mbox{for} \;
t \not=\frac{jT}{n}.
$$
\end{itemize}
\begin{theorem}
\label{theoremA}
The rectangular rule with $\Delta  = T/n$ is globally convergent:
for $t,r \in [0,T]$, $r<t$, 
$$
\|U(t,r)g - U_{n}(t,r)g\|_{X} \leq C \|g\|_{Y} \;(t-r)  
\sup_{s \in [0,T]} \|A(s) - A_{n}(s)\|_{Y,X}.
$$ 
If $t,r \in [t_{j-1}, t_{j}]$, this global estimate 
implies the rate of convergence of order
$o(\Delta )$.
\end{theorem}
{\bf Proof}: Consider the identity:
\begin{equation}
U(t,r)g -U_{n}(t,r)g = -\int_{r}^{t} U(t,s)[A(s) -
A_{n}(s)]U_{n}(s,r)g \; ds,
\label{identity}
\end{equation}
which follows from the differentiation of $-U(t,s)U_{n}(s,r)g$ with respect
to $s$, followed by its integration, after the conclusions of Theorem
\ref{UE} and the above observations have been introduced. 
The estimate is now immediate from the uniform convergence of the
generator sequence.
$\Box$
\subsection{Optimal or High Precision Quadrature} 
\label{B}
Although high-precision quadrature is much used (see \cite{LHP} for a
Crank-Nicolson evolution operator approximation),
its analysis via approximation theory, including the Bramble-Hilbert lemma
and the Sobolev representation theorem, appears minimal. 
The much older classical theory is described in \cite{SB}.
To fix the ideas, we consider the
method locally, as used on a subinterval originally defined via the
rectangular rule. The analysis is not restricted to Gaussian quadrature.
\begin{definition}
\label{Bdef}
The structure of the Hamiltonian here is assumed of the form written in
equation (\ref{eq_pot}), and $V$ has the meaning of a potential.
We introduce constants $c_{j}$, associated with $p$ interior points 
$t_{j}$ of an
interval $I$ of length $\Delta$, such that $\sum_{j=1}^{p} c_{j} f(t_{j})
\Delta$ is a quadrature approximation for $\int_{I} f(t) \; dt$. 
On the interval $[t_{0}, t_{0}+\Delta]$, define, for $s \leq t$, 
 \begin{equation}\label{global}
{\hat{U}}_{p}(t,s)=
\mathcal{T}\exp\left\{-(t-s)\sum_{j=1}^{p}\frac{i}{\hbar}c_j 
\hat {H}(t_j)\right\}.
\end{equation}
\end{definition}
We require the constants $c_{j}$
of the rule to reproduce the spatial part of the
operator. There are two parts of the error as seen from approximation
theory. There is that determined from the approximate evolution operators,
as induced by the quadrature. This is estimated in the following theorem.
However, there is also the initial error: that inherited by the quality of
the approximation of the solution at the beginning of the local time
interval. This is not an input directly controlled.

\begin{theorem}
\label{theoremB}
Suppose that ${\hat U}$ is invariant on the smooth Sobolev space: 
${\mathcal H} =  H^{4p}(\Omega)\cap H^{1}_{0}(\Omega)$,
and $V$ is smooth: $V \in C^{\infty}(I \times {\bar \Omega})$. 
If the quadrature scheme of Definition \ref{Bdef} 
has precision $2p-1$, 
then the evolution
operators constructed by the approximation scheme satisfy the estimate 
in $B[{\mathcal H}, L^{2}]$:
for any $g$ of norm one in ${\mathcal H}$,
$$
\|{\hat U}(t_{0} + \Delta,t_{0})g - {\hat U}_{p}(t_{0} + \Delta,t_{0})g\|_{X} 
\leq C(p, V)  \Delta^{2p}. 
$$ 
Here, $C(p,V)$ is proportional to a reciprocal Taylor factorial in $2p$; 
the supremum (over $\Omega$) of the $H^{2p}(I)$ norm of $V$ is the
dominant $V$-contribution.
\end{theorem}
{\bf Proof}:
We begin with (\ref{identity}), 
with a re-interpretation of 
$A(s) - A_{n}(s)$ as a difference of potentials:
$$
A(s) - A_{n}(s) \mapsto \frac{i}{\hbar} [V(\cdotp, s) - V_{p}(\cdotp)], 
$$
where $V_{p}$ is defined by $V_{p} = \sum_{j} c_{j} V(\cdotp, t_{j})$. 
We emphasize that the sum defining $V_{p}$ is to be taken as {\it time
ordered}.
We have used the reproduction of the spatial part of the operator by the
quadrature scheme in writing this reduction.
Thus, we have from (\ref{identity}), with $r \mapsto t_{0}, t \mapsto t_{0} +
\Delta$:
\begin{equation*}
{\hat U}(t_{0} + \Delta,t_{0})g -{\hat U}_{p}(t_{0} + \Delta,t_{0})g =
\end{equation*}
\begin{equation}
 -\frac{i}{\hbar}
\int_{t_{0}}^{t_{0} + \Delta} {\hat U}(t_{0} + \Delta,s)[V(\cdotp,s) -
V_{p}(\cdotp)]{\hat U}_{p}(s,t_{0})g \; ds.
\label{rep}
\end{equation}
We add and subtract the following 
quadrature  estimator function within the integrand of
(\ref{rep}): 
$${\mathcal Q}(\cdotp) = \sum_{j=1}^{p}\frac{i}{\hbar}c_j 
{\hat U}(t_0 + \Delta, t_{j})V(\cdotp, t_{j}))
{\hat U}_{p}(t_{j},t_{0})g.$$ 
This gives two terms, equivalent to quadrature estimation for two distinct
functions: 
\begin{equation*}
{\hat U}(t_{0} + \Delta,t_{0})g -{\hat U}_{p}(t_{0} + \Delta,t_{0})g =
\end{equation*}
\begin{equation*}
 -\frac{i}{\hbar}
\int_{t_{0}}^{t_{0} + \Delta} [{\hat U}(t_{0} + \Delta,s)V(\cdotp,s) 
{\hat U}_{p}(s,t_{0})g - {\mathcal Q}(\cdotp)] \; ds
\end{equation*}
\begin{equation}
 +\frac{i}{\hbar}
\int_{t_{0}}^{t_{0} + \Delta}[{\hat U}(t_{0} + \Delta, s) 
V_{p}(\cdotp){\hat U}_{p}(s,t_{0})g - {\mathcal Q}(\cdotp)] \; ds.
\label{twoterms}
\end{equation}
It remains to estimate the linear functionals summed above in
(\ref{twoterms}), and defined by
the difference of integration and quadrature evaluation in each case. 
Although the hypotheses of the Bramble-Hilbert Lemma \cite[Theorem 2]{BH} are
directly satisfied, the conclusion is not sufficient: 
this implies an order $O(\Delta^{2p})$ approximation multiplied by a time
integrated expression, involving the $2p$-th derivative of $V$.
To obtain a more precise error estimate, also involving the factorial, and
required here, we (additionally)
apply the Sobolev representation theorem (see \cite[Prop.
4.1.1]{J1}). This provides the full, triple product estimate, which
includes the (Taylor) factorial. 
Since this estimate is maintained with respect to integration over
$\Omega$, the proof is concluded.
$\Box$.
\subsubsection{Evaluation of the quadrature  rule approximation}
\label{C}
In practice, equation (\ref{oper1111}) offers an attractive numerical 
alternative to the original Magnus expansion \cite{Mag} when applied to large systems. 
The product of exponentials does not require the manipulation of commutators, and it 
can also be addressed very efficiently using our FEAST spectral approach (more details will follow 
in the simulation section).
Note that
the iteration of the semigroup exponentials in equation 
(\ref{oper1111}) represents a slight weighted version 
extension of the rectangular rule to unequally spaced nodes. 
One can adapt the proof of Theorem \ref{theoremA} to this case to obtain a
convergence order of $O(\delta)$. However, it does not seem possible to
improve this estimate to $o(\delta)$ as is possible in the case of the
rectangular rule. Note that (\ref{identity}) involves the difference
between the generator and the approximate generator; in the case of the
rectangular rule, this approximation converges uniformly in norm over the
$t$-interval. This does not appear to be the case for the exponential
product, where 
one cannot assert the local convergence of the generator approximation.
However, the program carried out in \cite{AF}, explicitly up to order
eight, proposes an interesting improvement:  
the definition of a `nearby' discrete problem, so that the so-called
commutator-free product exponential rule discussed here can be applied via
adjusted weights to improve convergence. It appears to be an open problem
as to the actual computational complexity associated with such improved
estimates. We note that, 
in the simulations of the
following sections (see Figure 1), one uses very high-order Gauss-Legendre
rules. Remarkably, one sees a very close relation between the
predictions of Theorem \ref{theoremB} and the actual numerical convergence.  

\subsection{Global in-time solution for admissible Hamiltonians}
We show in this section that a solution for the initial-value problem for
the nonlinear Schr\"{o}dinger equation exists for the admissible
Hamiltonians we have introduced in section \ref{admissible}. 
In addition to the regularity assumed for $V_{\rm ex}$ previously, we also
require here the existence and boundedness of its time derivative.
The exchange-correlation potential is not
included in this formulation.
We retain the meaning of $X,Y$ in this section, previously established in 
section \ref{admissible}.
\begin{definition}
For $J = [0, T], T$ arbitrary, 
define $K: C(J; X) \mapsto
C(J; X)$ by
$$
K \phi (\cdotp, t) = U^{\phi}(t, 0) \Psi_{0}, 
$$
where $U^{\phi}$ has been defined in Proposition \ref{generate} of the
appendix, and corresponds to the Hartree potential $W \ast |\phi|^{2}$.. 
\end{definition}
\begin{remark}
\label{U}
We will have need of estimates of 
$\|U^{\phi}(t,s)\|_{X}$ and $\|U^{\phi}(t,s)\|_{Y}$. 
On the $L^{2}$ space $X$, the operators preserve norm. 
On $Y$, the operators $U^{\phi}(t, s)$ have norm which is bounded from
above by a constant $C$ with dependency, 
$C(T, \|S\|_{Y,X}, \|S^{-1}\|_{X,Y}, \|\phi\|_{C(J; X)})$ (see
\cite[Cor.\ 6.3.6]{J1}).
\end{remark}
\begin{lemma}
The mapping $K$ is a compact and continuous mapping of
$Z = C(J; X)$ into itself. 
\end{lemma}
\begin{proof}
We prove the following, which is essential to both parts of the proof.
\begin{itemize}
\item
The image under $K$ of any ball ${\cal B} \subset Z$  
is an equicontinuous family: 
the distance 
$\|K \phi - K \psi \|_{Z} < \epsilon$ if  
$\|\phi - \psi \|_{Z} < \delta$.
Here, $\delta$ is independent of $\phi$, and $\psi$, provided 
$\phi$ and $\psi$ lie in a fixed ball of $Z$.
This is implied by: 
\begin{equation}
\label{identity2}
\|K \phi - K \psi \|_{Z} \leq C_{\psi} T 
\sup_{0 \leq s \leq T} \|A^{\phi}(s) - A^{\psi}(s)\|_{Y,X} \|\Psi_{0}\|_{Y}.
\end{equation}
\end{itemize}
Here, $C = C_{\psi}$ is the constant of Remark \ref{U}.
We use a variant of (\ref{identity}) to prove (\ref{identity2}): 
$$
U^{\phi}\Psi_{0}(t) - U^{\psi}\Psi_{0}(t)  
= -\int_{0}^{t} U^{\phi}(t,s)[A^{\phi}(s) -
A^{\psi}(s)]U^{\psi}(s,0)\Psi_{0} \; ds.
$$
A direct estimate implies (\ref{identity2}). We now show that this yields
the asserted equicontinuity. In particular, we must estimate 
$\sup_{0 \leq s \leq T} \|A^{\phi}(s) - A^{\psi}(s)\|_{Y,X}$, where the
operators $A^{\phi}, A^{\psi}$ are defined in Proposition \ref{generate},
via the Hamiltonians defined there. Clearly, the essential term is:
$$\|(W \ast |\phi|^{2})g - (W \ast |\psi|^{2})g\|_{X},$$ 
which is estimated by an adaptation of inequality (\ref{Hartineq}).
Specifically, we have:
\begin{equation}
\label{Hartineq2}
\|(W \ast |\phi|^{2})g - (W \ast |\psi|^{2})g\|_{X}| 
\leq \|W\|_{X} \||\phi|^{2} - |\psi|^{2} \|_{L^{1}}
\|g\|_{X},
\end{equation}
which, after factorization, 
is finally estimated via $\|\phi - \psi\|_{X}$, provided 
$\phi$ and $\psi$ lie in a fixed ball of $Z$.
We now prove the continuity and compactness.
\begin{itemize}
\item
The continuity of $K$ on $Z$.
\end{itemize}
This follows from the equicontinuity property proven above.
\begin{itemize}
\item
The compactness of $K$. 
\end{itemize}
This is more delicate. If ${\cal B}$ is bounded in $Z$,
we show that ${\cal K} = 
K {\cal B}$ is relatively compact in $Z$ by use of the
generalized Ascoli theorem \cite[Theorem 6.1, p.\ 290]{Munk}. 
This requires equicontinuity of the family ${\cal K}$, 
shown above. It also requires that 
$$
{\cal K}_{t} = \{u(t): u \in {\cal K} \}
$$ 
is relatively compact in $X$ for each $t \in J$. Since $Y$ is compactly
embedded in $X$, it is sufficient to show that ${\cal K}_{t}$ is bounded in
$Y$ for each $t \in J$. However, an application of Remark \ref{U}
immediately implies this. This concludes the proof. 
\end{proof}
\begin{theorem}
\label{global.existence}
The mapping $K$ has a fixed point $\Psi$. In particular,  
$$
i \hbar \frac{\partial \Psi(t)}{\partial t} = 
-\frac{\hbar^{2}}{2 m}\nabla^{2}\Psi + V_{\rm ex}\Psi 
+ (W \ast |\Psi|^{2}) \Psi,  
\; \Psi(\cdotp, 0) = \Psi_{0}.
$$
\end{theorem}
\begin{proof}
We use the Leray-Schauder theorem \cite{GT}.
Suppose $u = sKu$, for some $s$, where $0 < s \leq 1$. It is necessary to
establish a bound for $u$ in $Z$, which is independent of $s$; note that
$u$, in general, depends on $s$. It is easier to work with $\Psi = Ku$,
which satisfies the initial value problem:
\begin{equation}
\label{FPEq}
i \hbar \frac{\partial \Psi(t)}{\partial t} = 
-\frac{\hbar^{2}}{2 m}\nabla^{2}\Psi + V_{\rm ex}\Psi 
+ s^{2}(W \ast |\Psi|^{2}) \Psi,  
\; \Psi(\cdotp, 0) = \Psi_{0}.
\end{equation}
The technique we use is 
conservation of energy, formulated to include the external potential. 
We establish the following:
\begin{itemize}
\item
If the energy is defined 
for $0 < t \leq T$ by,  
$$
E_{s}(t)=\int_{\Omega}\left[\frac{{\hbar}^{2}}{4m}|\nabla \Psi|^{2} 
+ 
\left(\frac{s^{2}}{4}(W \ast |\Psi|^{2})+ 
V_{\rm ex}\right)|\Psi|^{2}\right]\;dx_{1}dx_{2}dx_{3},
$$
then the following identity holds:
\end{itemize} 
\begin{equation}
E_{s}(t)=E_{s}(0)
+\int_{0}^{t}\int_{\Omega}(\partial V_{\rm ex}/\partial r)({\bf x},r)
|\Psi|^{2}\;dx_{1}dx_{2}dx_{3}dr,
\label{consener}
\end{equation}
where 
$$E_{s}(0) =  
\int_{\Omega}\left[\frac{{\hbar}^{2}}{4m}|\nabla
\Psi_{0}|^{2}+\left(\frac{s^{2}}{4} 
(W\ast|\Psi_{0}|^{2})+V_{\rm ex}\right)|\Psi_{0}|^{2}\right]
\;dx_{1}dx_{2}dx_{3}.
$$
We first observe that (\ref{consener}) is sufficient to imply that the
functions $\{\Psi\}$, and hence the functions $\{u\}$, are bounded in $Z$; 
indeed, $L^{2}$ gradient bounds for $\Psi$ are obtained from $E_{s}(t)$.
These bounds 
depend only on $\Psi_{0}, V_{\rm ex}$, and the time derivative  
of $V_{\rm ex}$. Note that 
$\Psi$ has $X$ norm equal to that of $\Psi_{0}$.

It remains to verify (\ref{consener}); in fact, we establish its
derivative:
\begin{equation}
\label{consenerder}
0 = \frac{dE_{s}}{dt} - 
\int_{\Omega}(\partial V_{\rm ex}/\partial t)({\bf x},t)
|\Psi|^{2}\;dx_{1}dx_{2}dx_{3}.
\end{equation}
We use (\ref{FPEq}): multiply by $\partial {\bar \Psi}/\partial t$,
integrate over $\Omega$, 
and take the real part. This is a standard technique and yields
(\ref{consenerder}). This concludes the proof. 
\end{proof}

\section{Numerical Simulations and Discussions}
In this section, we propose to illustrate the validity of the Theorems 
\ref{theoremA} and  \ref{theoremB} using a selected pair of realistic 
numerical experiments (these examples are not restrictive). 
The first example in \ref{sec_ex1} considers the real-time
propagation of the Kohn-Sham wave functions with an external potential 
$V({\bf x},t)$ 
which is linear in the quantum state. The second example in \ref{sec_ex2}
presents the TDDFT real-time propagation model
 within the adiabatic local density approximation (ALDA), where the potential, 
which includes both the Hartree and exchange-correlation
terms, is non-linear in the quantum state but local in time. In the 
following, we first describe some elements of the numerical modeling strategy 
that have been used in both examples including the finite-element 
discretization, the spectral decomposition of the evolution operator, and 
the FEAST eigenvalue algorithm.

\subsection{Numerical modeling}

For a system composed of $N_e$ electrons,  the ground state electron density 
$n(\textbf{r})=2\sum_i^{N_e}|\psi_i(\textbf{r})|^2$
 (i.e. $2$ for the spin factor)
can be obtained from the solution of the DFT Kohn Sham stationary equation 
\cite{KS}:
\begin{equation}\label{ks}
\left[ -\frac{\hbar^2}{2m}\nabla^2+v_{KS}[n](\textbf{r}) 
\right]\psi_j(\textbf{r})= E_j \psi_j(\textbf{r}),
\end{equation}
where the Kohn-Sham potential $v_{KS}$ is a functional of the density and it is 
conventionally separated in the following way:
\begin{equation}
\label{eqvks}
v_{KS}[n](\textbf{r})=v_{ext}(\textbf{r})+v_{ion}(\textbf{r})
+v_{H}[n](\textbf{r})+v_{xc}[n](\textbf{r}),
\end{equation}
where $v_{ext}$ is the external potential; $v_{ion}$ is the ionic or 
core potential;  $v_H$  is the Hartree potential which 
accounts for the electrostatic interaction between the electrons and 
is the solution of a Poisson equation; and $v_{xc}$ is the 
exchange-correlation potential which accounts for all the non-trivial 
many-body effects.

In TDDFT, all the $N_e$ initial wave functions 
$\Psi=\{\psi_1,\psi_2,\hdots,\psi_{N_e}\}$, which are  
solutions of the Kohn-Sham system (\ref{ks}), are propagated in time using 
a time-dependent Schr\"odinger equation: 
\begin{equation}\label{kst}
i\hbar\frac{\partial}{\partial t}\psi_j(\textbf{r},t)=
\left[-\frac{\hbar^2}{2m}\nabla^2+
v_{KS}[n](\textbf{r},t)\right]\psi_j(\textbf{r},t), \quad
\forall j=1,\dots,N_e.\end{equation}
The electron density of the interacting system can then be obtained at 
a given time from the time-dependent Kohn-Sham wave functions
\begin{equation}
n(\textbf{r},t)=2\sum_{j=1}^{N_e} |\psi_j(\textbf{r},t)|^2.\end{equation}
In our numerical experiments, we consider the ALDA approach 
where the exchange-correlation potential $v_{xc}$ in (\ref{eqvks}) 
depends locally on time and 
it is a functional of the local density $n(\textbf{r},t)$ i.e.
\begin{equation}
\label{eqvkst}
v_{KS}(n(\textbf{r},t))=v_{ext}(\textbf{r},t)
+v_{ion}(\textbf{r})+v_{H}(n(\textbf{r},t))+v_{xc}(n(\textbf{r},t)).
\end{equation}
As discussed in this article, we consider the integral form of (\ref{kst}) 
defined in (\ref{prop}) along with the time-discretization of the 
evolution operator given in (\ref{oper111}).

The discretization of the Hamiltonian operator in real-space is performed
using the finite element method. 
For the choice of the elements, we consider respectively prisms 
for example 1 and tetrahedra for example 2, 
 using either quadratic or cubic precision. If $\mathbf{H}$ 
denotes the resulting $N\times N$ Hamiltonian matrix 
at a given time $t$ and if $N$ represents the number of finite-element
nodes, the spectral decomposition of $\mathbf{H}$ can be written as follows:
\begin{equation}\label{eig}
\mathbf{D}(t)=\mathbf{P}^{T}_t \mathbf{H}(t)\mathbf{P}_t,
\end{equation}
where the columns of the matrix $\mathbf{P}_t$ 
represent the eigenvectors of $\mathbf{H}(t)$ associated with the 
eigenvalues regrouped within the diagonal matrix $\mathbf{D}(t)$. 
Since the $N_e$ propagated states are low-energy states, it is reasonable 
to obtain very accurate spectral approximations, even by using a partial 
spectral decomposition, where
one considers a number $M$ of lowest eigenpairs much smaller than the size
$N$ of the system but greater than $N_e$ (i.e. $N_e<M<<N$).
In all of our numerical experiments, increasing the value of our choice for 
$M$ has had no influence on the stability of the solutions.
The exact error analysis introduced in this spectral decomposition is 
proposed as future work in section \ref{sec8}.

Since the discretization is performed using non-orthogonal basis functions 
(e.g. finite element basis functions), the eigenvalue 
problem that needs to be solved at given time $t$ takes the generalized form:
\begin{equation}\label{eig1}
\mathbf{H}(t)\mathbf{p_i}(t)=\mathrm{d_i}(t)\mathbf{S}\mathbf{p_i}(t),
\end{equation}
where $\mathbf {S}$ is a symmetric positive-definite matrix, and the 
eigenvectors ${\bf p_i}(t)$ 
are  $\mathbf{S}$-orthonormal i.e.  
$\mathbf{P}^{T}_t\mathbf{S}\mathbf{P}_t=
\mathbf{I}$ with $\mathbf{P}_t=\{{\bf p_1}(t),{\bf p_2}(t),\dots,{\bf p_M}(t)\}$.
By use of the spectral decomposition of the Hamiltonian (\ref{eig}), 
the exponential in (\ref{oper1111})
acts only on the eigenvalue matrix $\mathbf{D}(t)$, 
and one can show that the resulting matrix form of the time propagation 
equation is given by:
\begin{equation}\label{propd}
\mathbf{\Psi}(t+\Delta_t)=
\mathcal{T}\left\{ \prod_{j=1}^p\left[ \mathbf{P}_{t_j} 
\exp\left(-\frac{i}{\hbar}\xi\omega_j
\mathbf{D}(t_j)\right) \mathbf{P}^T_{t_j} \mathbf{S} \right]   \right\} 
\mathbf{\Psi}(t).
\end{equation}

In real large-scale applications, a direct solution of the evolution operator 
has often been considered impractical, since it
requires solving a hundred to a thousand eigenvalue problems 
along the time domain 
(one eigenvalue problem for each time step). 
However, we rely on the capabilities of the new 
FEAST eigenvalue algorithm  \cite{POL} and solver \cite{feast}, which is 
ideally suited for addressing such challenging calculations. 
FEAST is a general purpose algorithm for obtaining selected eigenpairs 
within a given search interval.
It consists of integrating the solutions of very few independent linear systems
for the Green's function 
\mbox{${{\bf G}({Z})}= ({Z}{\bf S}-{\bf H})^{-1}$} of size $N$ along 
a complex contour (typically $8$ to $16$ contour points 
by use of a Gauss-Legendre quadrature), and 
 one reduced dense eigenvalue problem arising from a Raleigh-Ritz 
procedure (of size $M_0\simeq1.5M$ in the present case).
FEAST relies also on a subspace-iteration procedure where convergence is 
often reached in 
very few iterations ($\sim 3$) to obtain up to thousands of eigenpairs with 
machine accuracy. 
An efficient parallel implementation can be addressed  at three different 
levels ranging from 
the selection of the search intervals, to solving independently the inner 
linear systems along with their own parallel treatment.
As a result, the algorithm complexity for performing the spectral 
decomposition (\ref{eig}) is directly dependent on solving a single 
complex linear system of size $N$. 
In comparison with a Crank-Nicolson scheme where 
small time intervals are needed and the linear systems need to be solved one after another, 
the spectral approach allows for larger time intervals and a parallel implementation of FEAST
 requires only one linear system to be solved by interval. It is important to note that even if $M$
becomes very large, linear parallel scalability can still be obtained using
multiple contour intervals and an appropriate parallel computing power.
Finally, FEAST is also ideally suited for addressing 
efficiently the time propagation equation in (\ref{propd}), 
since it can take advantage of the subspace computed at a given time 
step $j$ as initial
guess for the next time step ${j+1}$ in order to speed-up the 
numerical convergence.

\subsection{Example 1}\label{sec_ex1}

We consider the real-time propagation of the Kohn-Sham quantum states for a 
Carbon nanotube (CNT) device in interaction with an electromagnetic (EM) THz 
radiation \cite{CP}.
In this example, the three dimensional time dependent potential (\ref{eqvkst}) 
does not depend on the electron density and
takes the following form:
$$ v_{KS}({\bf r},t)=v_{eps}({\bf r})+v_{ext}({\bf r},t),$$
where $v_{eps}$ is a time-independent atomistic empirical pseudopotential which
approximates the effect of $v_{ion}$, $v_H$ and $v_{xc}$ at $t=0$, 
and $v_{ext}$ is a time-dependent external potential applied along the 
longitudinal direction $x$ of the CNT, i.e.,  
$v_{ext}=v_{0}((2x-L)/L)sin(\omega t)$ with $x\in[0,L]$, which leads to 
a constant electric field along the direction of $x$. 
For performing the 3D simulations, we consider 6 unit cells of a (5,5) CNT 
with length $L=1.98$nm, $v_{0} = 5$eV, and  $\omega=2\pi f$, with $f=200$ THz.
In our simulations, all the solution wave functions $\bf \Psi$ (\ref{propd}) 
will be propagated from
$t = 0$ to $t = 8T$, where $T = 1/f = 5\times 10^{-15} s$ 
denotes the period of the EM radiation.
Figure \ref{fig1} provides the time evolution of the energy expectation 
for the highest occupied molecular orbital 
(HOMO level) by using both (i) the rectangular
rule and small time-step intervals $\Delta=T/p$, and 
(ii) the high-order integration rule with a long time-step interval 
$\Delta=T$ and $p$ interior points.
\begin{figure}[htbp]
\centering
\includegraphics[width=1.0\linewidth,angle=0]{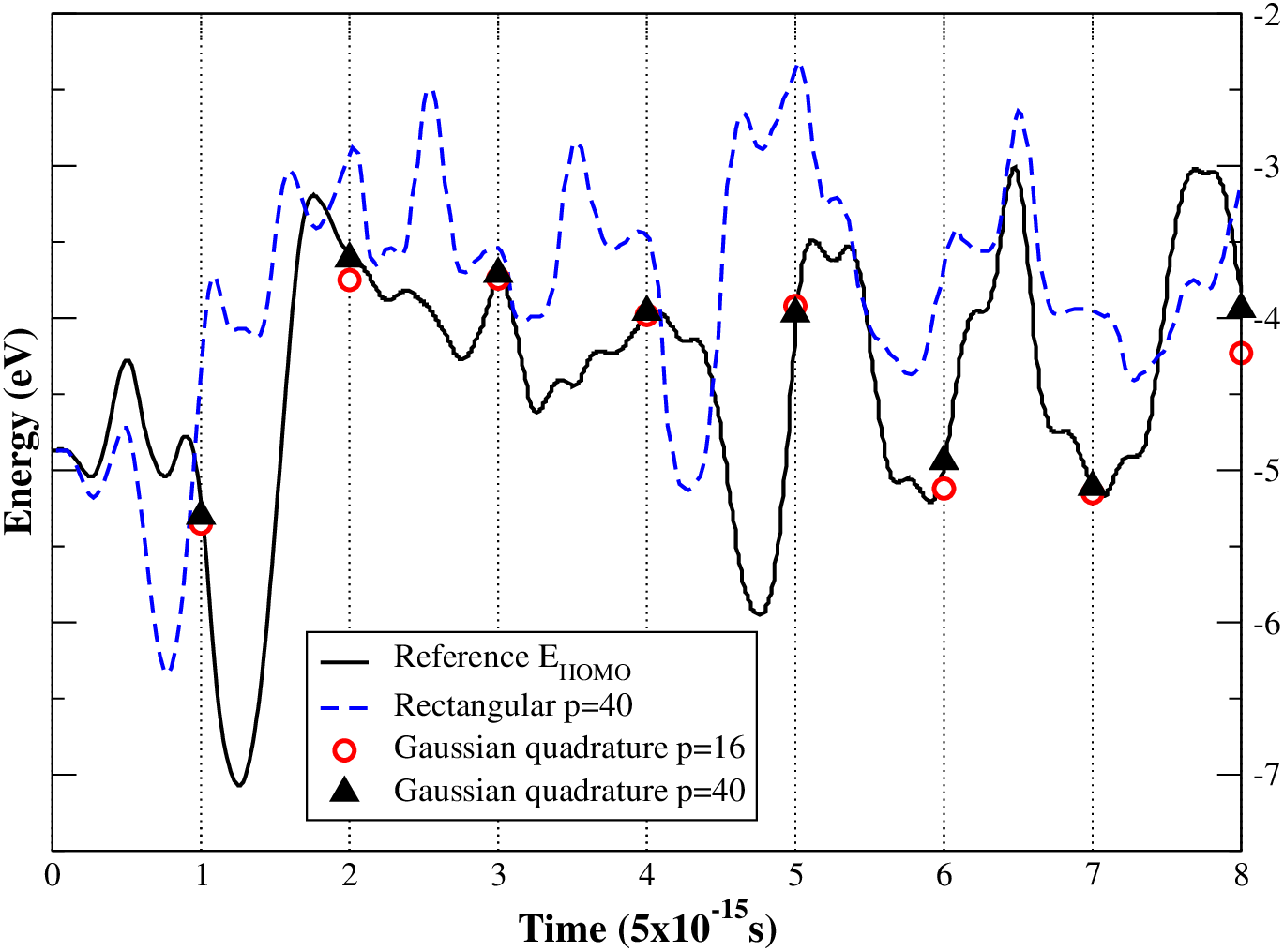}
\caption{\label{fig1} Evolution of the energy expectation of the HOMO level 
along $8$ time
periods of the EM THz radiation (i.e. $E(t)={\bf \Psi^\dagger}_{N_e}(t) {\bf H}(t) {\bf \Psi}_{N_e}(t)$). 
The solid lines represent the
reference solution. The result for the energy evolution obtained by using a 
rectangular rule  with $p = 40$  
diverges after a few time steps. The same number of interior points, however, 
is adequate to capture the solutions accurately at the end of each
time period by using the Gauss quadrature scheme.
The solutions obtained using the $p = 16$ Gauss scheme begin to
be affected by 
the approximation constructed from 
the decomposition of the exponential (\ref{oper1111}) 
due to an increase in distance between integration points. 
We note that the intermediate solutions obtained using the 
Gauss integration rule have no
physical meaning, and are not then represented here.}
\end{figure}
The reference solutions have been obtained using the rectangular rule and 
$p=120$, where the solution has converged.
Using a rectangular rule with $p = 40$ integration points by period, 
one notes that the predicted 
results begin to diverge after a few time steps, and 
this phenomenon amplifies with time. 
From Theorem \ref{theoremA}, it is necessary to increase $p$ 
(i.e., decrease the time-step interval) to improve the convergence rate
of the rectangular approximation. In contrast,  
one can see from the numerical results that $p = 40$ interior points,  
by using a high-order Gauss 
integration scheme, does suffice to 
obtain the solution accurately at each long-time interval $\Delta$. 
This result can be justified by 
Theorem \ref{theoremB}. 
By decreasing the number of interior points $p$ even further, 
it is expected to obtain  a lower order of approximation.  
One can indeed 
confirm a lower convergence rate from the numerical results of the Gauss 
integration scheme using  $p=16$ interior points.
For these lower values of $p$, one finds convergence comparable to the 
steps of  
the exponential product evaluation (\ref{oper1111}). 
This emphasizes the higher order convergence of the integration 
rule (see section \ref{C}).

\subsection{Example 2}\label{sec_ex2}

This example focuses on obtaining 
the evolution of the time-dependent dipole moment of the CO molecule 
by using a real-time propagation approach and the non-linear
TDDFT-ALDA model. We follow a similar procedure to that presented in \cite{YB} 
where once the ground-state DFT density and the Kohn-Sham states are obtained, 
a short polarized impulse is applied along the longitudinal or perpendicular 
direction of the molecule. If $z$ denotes the perpendicular 
direction of the molecule, after a short delta impulse along $z$, 
the Kohn-Sham states (\ref{ks}) are modified as follows:
\begin{equation}
 \psi_j(\textbf{r},t=0^+)= \exp(-\imath Iz/\hbar) \psi_j(\textbf{r},t=0),
\end{equation}
where $I$ is the magnitude of the electric field impulse.
Thereafter, equation (\ref{propd}) is solved by using a non-linear potential 
(\ref{eqvkst}) and no external perturbation (i.e. $v_{ext}=0$).
It is also important to note that  our simulations are performed by using 
an {\em all-electron} model since the
 potential $v_{ion}$ is not approximated and includes the core potentials.
The density obtained at each time step $\Delta_t$, is used to compute the 
induced dipole of the system:
\begin{equation}
 D(t)= \int_{\Omega} {\mathbf r}\left(n(\textbf{r},t)
- n(\textbf{r},0)\right) {\bf dr},
\end{equation}
which is relative to the center of mass of the molecule. 
$D(t)$ is a quantity of interest since the imaginary part of its Fourier
transform provides the dipole strength function and, for this example, 
the optical absorption spectrum along with
 the true many-body excited energy levels. Since we are investigating the 
optical frequency response rather than the THz
response presented in the first example, the time intervals are here chosen 
relatively shorter. 
Figure \ref{fig2} presents the time evolution of the dipole $D(t)$ obtained 
by using both 
a rectangular rule  with a time-step $\Delta_t=1\times10^{-17}$s and a 
Gauss quadrature scheme using $p=1$ 
and a time interval $\Delta_t=2\times10^{-17}$s. For the Gauss-scheme, 
a matrix exponential has to be evaluated at the middle of the interval
$[t,t+\Delta_t]$ for each $t+(1\times10^{-17}$)s (i.e. Gauss-1 presents 
only one node in the middle of the interval).
\begin{figure}[htbp]
\centering
\includegraphics[width=1.0\linewidth,angle=0]{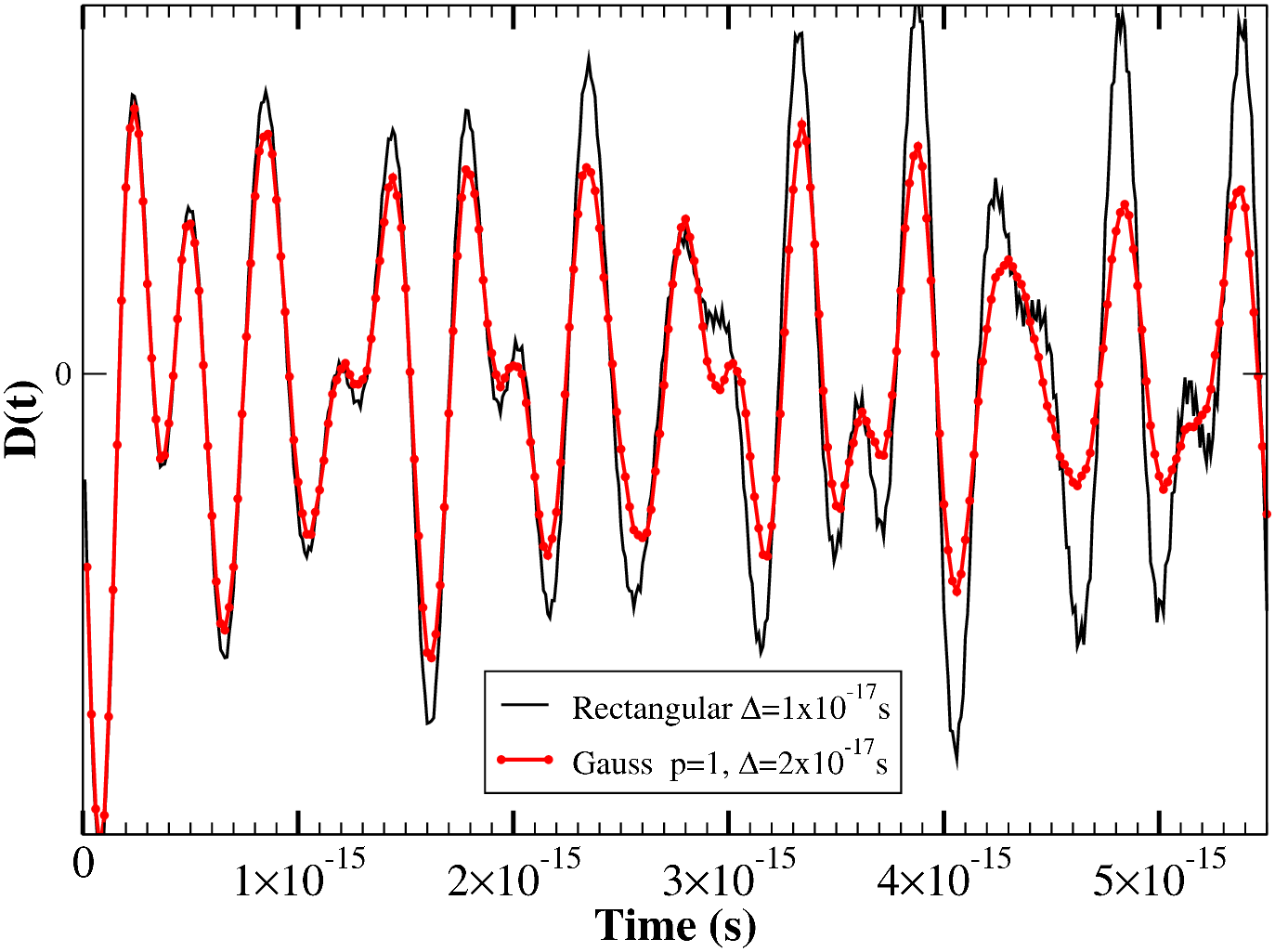}
\caption{\label{fig2} Time evolution of induced dipole moment of the CO 
molecule after a short-impulse along the perpendicular direction. 
 The results for the rectangular rule (i.e. direct approach) are obtained 
by using the time step $1\times 10^{-17}$ while
a time step of $2\times 10^{-17}$ is considered for the Gauss quadrature 
scheme with $p=1$. This latter scheme presents then only one interior point 
in the middle of the interval. Although the variation for the two responses 
presents a similar pattern, the Gauss-1's curve appears much smoother 
than the curve obtained using the  rectangular approximation.}
\end{figure}
One notes that both curves are identical at the early stage of the 
time evolution, and present also a similar frequency pattern behavior. 
However, the rectangular approximation clearly presents a staircase 
pattern which can be attenuated by using
a much smaller time-step as shown in Figure \ref{fig3}. 
\begin{figure}[htbp]
\centering
\includegraphics[width=1.0\linewidth,angle=0]{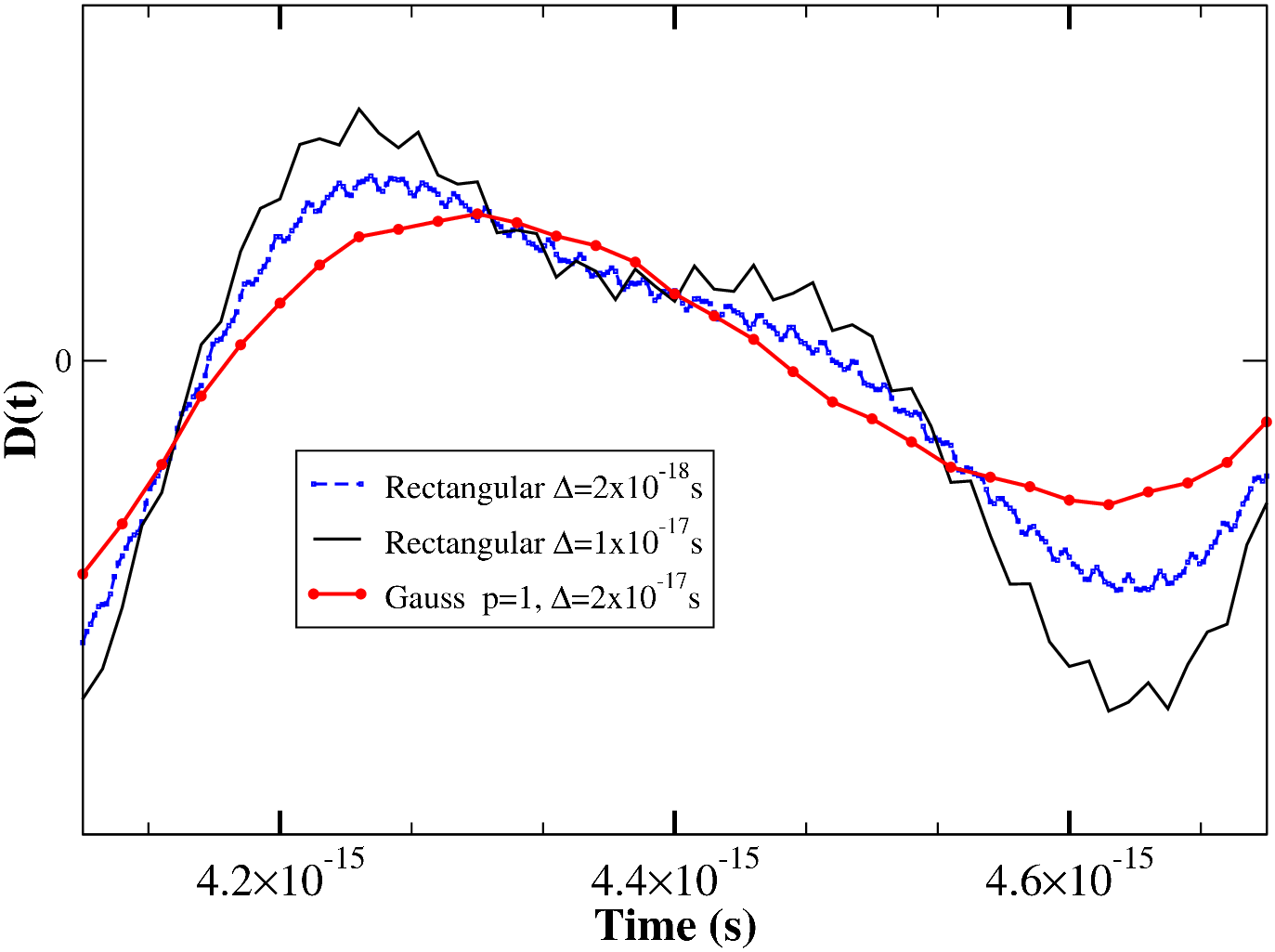}
\caption{\label{fig3}  Time-evolution sample of induced dipole moment of the CO
 molecule after a short-impulse along the perpendicular direction. 
These results are identical to the one presented in Figure (\ref{fig2}) 
over a selected period in time; 
the result for a rectangular rule using a much smaller time step, 
$\Delta_t=2\times 10^{-18}$, is also provided for comparison.
One notes that the rectangular approximation benefits from shorter time 
intervals since the resulting curve is both smoother
than the one obtained using $\Delta_t=1\times 10^{-17}$, and closer to the 
result obtained using Gauss-1.}
\end{figure}
This variation pattern is most likely related to the $P^0$ approximation 
used by the rectangular approximation
between intervals, while the Gauss-1 scheme is associated with a 
$P^1$ approximation.
It should be noted, however, that the rectangular rule provides a direct 
propagation scheme where the potential is always known in advance 
at a given time step $t$. In contrast, the Gauss scheme requires an a priori 
evaluation of the unknown potential 
 at the $p$ Gauss interior points.  
In practice, it is possible to use different extrapolation or 
predictor/corrector schemes, although
the overall procedure should ideally be self-consistent \cite{TDDFTbook}. 
All the models presented in this article where
 the potential is required to be known beforehand in the time interval, 
would then remain valid if such a self-consistent iterative procedure occurs.
The current example does not take advantage of the self-consistent procedure, 
and we have then considered only the use of the Gauss-1 propagation 
scheme where the potential/density is obtained beforehand in the middle of 
the interval.
Interestingly, the Gauss-1 model is known as the exponential mid-point rule by 
the TDDFT community \cite{CMR}, a robust scheme that also preserves
the time-reversal symmetry. Our general analysis offers here some
perspectives 
to go beyond the Gauss-1 scheme in order to provide 
more accuracy by using longer time-intervals. More details on the absorption 
spectrum and other physical results obtained by using the 
Gauss-propagating scheme will be provided elsewhere \cite{CP2}.


\section{Future Work and Perspectives}\label{sec8}
We have presented and analyzed a numerically efficient discretization approach 
for the general problem of real time propagation 
of the time-dependent evolution operator. Modern 
``matrix diagonalization techniques'', such as FEAST \cite{POL}, permit
the creation of
new methodologies for real-time propagation of large-scale quantum 
systems using direct integration and discretization of the 
time-ordered evolution operator.
As shown, it is also possible to define an approach that allows a  
significant reduction of
the number of eigenvalue problems which are solved in the time-stepping,
due to the smaller number of time step quadratures. 
Both the number and solution accuracy of these eigenvalue problems
contribute to the computational complexity in TDDFT.
\begin{itemize}
\item
Implicit in the time discretization is the further spectral approximation of
the evolution operators, inherent in the eigenvalue/eigenvector
calculations. An exact error analysis will incorporate both of these
features.
In terms of approximation theory, the error introduced by the spectral
approximations should be balanced by that of the time discretization. 
\end{itemize}
A detailed analysis of the numerical linear algebra 
of this spectral approximation 
step may be found in \cite{CP}, which appears to be 
one of the first instances of correlation of this type.
If time discretization is measured by the rectangular rule, this permits the
flexibility of lower-dimensional intermediate spectral approximation as
discussed, and implemented, in \cite{CP}. 
One can note that the techniques proposed here will be extremely efficient for 
linear physical systems using very large time-intervals.
The traditional notions of interacting and non-interacting systems in 
quantum physics are often used within the 
context of the single electron picture. Within TDDFT, the many body
problem becomes numerically tractable, but also non-linear with respect
to the 
electron density (i.e., interacting system). Since 
the electron density exhibits much weaker variations as compared to the 
variations 
of the individual wave functions, it would then become advantageous to use
time-intervals that are capable of capturing 
the variation of the electron density with time, while still being much 
longer than
the traditional short-time steps of rectangular approximations. 
Convergence analysis of a fully non-linear scheme would represent
an important step in TDDFT, and it is also a component of our future work.
However, local existence is much easier: 
\begin{itemize}
\item
The framework presented here establishes the existence of local in time
solutions to certain non-linear TDDFT systems, formulated for closed
systems. 
\end{itemize}
This is a consequence of Kato's theory 
(cf.\thinspace \cite{HKM}, \cite[Th. 7.2.4]{J1}), based on the contraction
mapping theorem.
We simply state, in summary, the character of the result, when the
potential is perturbed by a nonlinear function $\phi(\rho)$ of the charge.
As long as $\phi$ is bounded, with bounded derivatives, then 
the hypothesis of Theorem \ref{theorem 3.2} holds as long as $Y$ is
identified with a Sobolev space of sufficiently high index $s > 5/2$.
In this case, the isomorphism from $Y$ to $X$ is
implemented by intermediate spaces, described in \cite[pp.\thinspace
244--247]{Aubin}. The other hypotheses of the local existence theorem are
routine within the framework developed here.
These results are consistent with those obtained by other methods for
nonlinear Schr\"{o}dinger equations \cite{Caz}. 
In particular, it includes the case of the  exchange-correlation potential.
In future work, we will aim for establishing a global in time existence
theory, via the evolution operators,  
for non-linear TDDFT systems, which extends the applicability of
Theorem \ref{global.existence} and is also consistent with 
the literature \cite{Caz}. 

\section*{Acknowledgments}  
The second author is supported by the National Science Foundation under 
grants No ECCS 0846457 and No ECCS 1028510.

\appendix

\section{Time-Ordered Evolution Operators}
Time dependent quantum mechanics is ideally suited to the use of 
Kato's evolution operators, introduced in \cite{K1,K2}, improved in
\cite{D}, and summarized in detail in \cite[Ch.\ 6]{J1}. We present a
concise summary here, coupled to the hypotheses discussed earlier.
\subsection{Defining properties on the frame space $X$}
We briefly summarize the result. We begin with a complex Banach
space $X$ and denote by $G(X)$ the family of negative generators of
$C_{0}$-semigroups on $X$. We discuss the general case in this section;
the case of the Hamiltonian is retrieved by $A(t) \mapsto (i/\hbar)\hat H(t)$. 
\begin{definition}
\label{def3.1}
If a family $A(t) \in G(X)$ is given on $0 \leq t \leq T$, the family is
stable if there are stability constants $M, \omega$ such that 
\begin{equation}
\| \prod_{j = 1}^{k} [A(t_{j}) + \lambda]^{-1} \| \leq M
(\lambda - \omega)^{-k}, \;\; \mbox{for} \; \lambda > \omega,
\end{equation}
for any finite family $\{t_{j}\}_{j=1}^{k}$, with $0 \leq t_{1} \leq \dots
\leq t_{k} \leq T$. Moreover, $\prod$ is time-ordered:  
$[A(t_{\ell}) + \lambda]^{-1}$ is to the left of 
$[A(t_{j}) + \lambda]^{-1}$ if $\ell > j$. 
If $Y$ is densely and continuously embedded in $X$, 
and $A \in G(X)$, $Y$ is $A$-admissible if $\{e^{-tA}\}|_{Y}$ is 
invariant, and forms a 
$C_{0}$-semigroup on $Y$. 
\end{definition}

\smallskip
These are the preconditions for the theorem on the unique existence of the
evolution operators. 
\subsection{The general theorem for the frame space}
The following theorem concatenates \cite[Theorem 6.2.5, Proposition
6.2.7]{J1}.
\begin{theorem}
\label{UE}
Let $X$ and $Y$ be Banach spaces such that $Y$ is densely and continuously
embedded in $X$. Let $A(t) \in G(X), 0 \leq t \leq T$ and assume the
following.
\begin{enumerate} 
\item
The family $\{A(t)\}$ is stable with stability index
$(M, \omega)$.
\item
The space $Y$ is $A(t)$-admissible for each $t$. The family of generators
on $Y$ is assumed stable. 
\item
The space $Y \subset D_{A(t)}$ and the mapping $t \mapsto A(t)$ is
continuous from $[0, T]$ to the normed space $B[Y, X]$ of bounded linear
operators from $Y$ to $X$.
\end{enumerate}
Under these conditions the evolution operators $U(t,s)$ exist uniquely 
as bounded linear operators on $X$, $0 \leq s \leq t \leq T$ with the
following properties.
\begin{itemize}
\item[I]
The family $\{U(t,s)\}$ is strongly continuous on $X$, jointly in $(t,s)$,
with:
$$
U(s,s) = I, \; \|U(t,s)\|_{X} \leq M \exp[\omega (t-s)].
$$
\item[II]
The time ordering is expressed by:
$$
U(t,r) = U(t,s)U(s,r).
$$ 
\item[III]
If $D^{+}_{t}$ denotes the right derivative in the strong sense, then
$$
[D^{+}_{t}U(t,s)g]_{t=s} = -A(s) g, \; g \in Y, \; 0 \leq s < T.
$$
\item[IV]
If $d/ds$ denotes the two-sided derivative in the strong sense, then
$$
(d/ds) U(t,s)g = U(t,s)A(s)g, \;  
g \in Y, \; 0 \leq s \leq t \leq T.
$$
This is understood as one-sided if $s=t$ or $s = 0$.
\end{itemize}
\end{theorem}
An important question in the theory is what condition guarantees that the
evolution operators remain invariant on the smooth space $Y$. This is now
addressed.
\subsection{A result for the smooth space: regularity}
We quote a slightly restricted version of \cite[Theorem 6.3.5]{J1}.
\begin{theorem}
\label{theorem 3.2}
Suppose hypotheses (1,3) of Theorem \ref{UE} hold, and that  
there is an isomorphism $S$ of $Y$ onto $X$ such that
$$
SA(t)S^{-1} = A(t) + B(t), \; B(t) \in B[X], 
$$
a.e. on $[0,T]$, where $B(\cdotp)$ is strongly measurable, and 
$\|B(t)\|$ is Lebesgue integrable. 
Then hypothesis (2) of Theorem \ref{UE} holds. Also, the following hold.
\begin{itemize}
\item[I$^{\prime}$]
Invariance:
$$
U(t,s)Y \subset Y, \; 0 \leq s \leq t \leq T.
$$
\item[II$^{\prime}$]
The operator function $U(t,s)$ is strongly continuous on $Y$, jointly in
$s$ and $t$.
\item[III$^{\prime}$]
For each $g \in Y$, 
$$
(d/dt) U(t,s) = -A(t) U(t,s)g, \; 0\leq s \leq t \leq T, s<T.
$$
This derivative is continuous on $X$.
\end{itemize}
\end{theorem}
\subsection{The initial value problem}
The evolution operators permit the solution of the linear Cauchy problem,
\begin{eqnarray}
\frac{du}{dt} + A(t) u(t) &=& F(t), \\
u(0) = u_{0},
\label{CP}
\end{eqnarray}
on an interval $[0, T]$, with values in a Banach space $X$. 
The formal solution, 
\begin{equation}
u(t) = U(t,0) u_{0} + \int_{0}^{t} U(t,s) \; F(s) \; ds,
\label{SCP}
\end{equation}
holds rigorously 
under assumptions on $u_{0}, F$ (for a precise statement, cf.\thinspace 
\cite[Prop. 6.4.1]{J1}).
In particular, the initial-value problem 
(\ref{eeq},\ref{ic}) is solved by the identifications
$u \mapsto \Psi, A \mapsto \frac{i}{\hbar}\hat H$, with $F = 0$.

\subsection{Admissibility of the Hartree potential}
\begin{proposition}
\label{generate}
The operators,
$$
{\hat H}^{u}(t) =  
-\frac{\hbar^{2}}{2 m}\nabla^{2}+ V_{\rm ex}(\cdotp, t) +
W \ast |u(\cdotp, t)|^{2},
$$
with domain, $H^{2}(\Omega) \cap H^{1}_{0}(\Omega)$, satisfy the
hypotheses of Theorem \ref{theorem 3.2} for each 
$$
u \in C([0,T]; L^{2}(\Omega)).
$$
Here, $T$ is an arbitrary terminal time and the identifications, 
$$
X = L^{2}(\Omega), \; Y = 
H^{2}(\Omega) \cap H^{1}_{0}(\Omega), 
$$ 
are made. The external potential $V_{\rm ex}$ is assumed continuous from
the time interval into the space  
of twice continuously differentiable functions, 
with bounded derivatives through order two in ${\bf x}$.
In particular, the evolution operators $U^{u}(s,t)$ exist in the sense of
Theorem \ref{theorem 3.2} when the identification $A^{u}(t) =
(i/\hbar){\hat H}^{u}(t)$ is made.
\end{proposition}
\begin{proof}
The proof proceeds by verifying hypotheses 1,3 of Theorem \ref{UE}
and the similarity hypothesis of Theorem \ref{theorem 3.2}.
It is equivalent to use ${\hat H}^{u}(t)$.
We use $X,Y$ for notational convenience in the proof.
We note the inequality, for each $t$,
\begin{equation}
\label{Hartineq}
\|(W \ast |u|^{2})g\|_{X} \leq \|W\|_{X} \||u|^{2} \|_{L^{1}}
\|g\|_{X},
\end{equation}
which follows from the Schwarz inequality and Young's inequality.
This implies that the Hartree potential defines a bounded linear operator
on $L^{2}$ for each $t$. The same is true for $V_{\rm ex}$. 
This permits the straightforward verification of the Assumption in 
section \ref{Ham} for the operators 
$\{{\hat H}^{u}(t)\}$.  In fact, one can employ the Friedrichs extension to the
symmetric operator defined on 
infinitely differentiable compact support functions. 
The above inequality and the assumed properties of $V_{\rm ex}$ 
can be used to verify the third hypothesis of
Theorem \ref{UE}. 
It remains to verify the similarity relation expressed in Theorem
\ref{theorem 3.2}. We define $S^{-1}$ here as the `solver' for the homogeneous
Dirichlet problem for the (negative) Laplacian $S$ on $\Omega$; 
the boundary is assumed
sufficiently smooth to allow for $H^{2}$ regularity for the solver when
applied to $L^{2}$ functions. 
By direct calculation
we have:
$$
S{\hat H}^{u}S^{-1} g = -\frac{{\hbar}^{2}}{2m} \nabla^{2}g + 
SV_{\rm ex}S^{-1}g + S (W \ast |u|^{2}) S^{-1} g, 
$$
for $g \in Y$.
It is necessary to demonstrate that the second and third operators are
bounded on $L^{2}$ for each $t$. For the third operator, one has
$$
B(t)g = 
S (W \ast |u|^{2}) S^{-1} g = 4 \pi |u|^{2} S^{-1}g - 2 \sum_{j = 1}^{3} 
(W_{x_{j}}\ast |u|^{2})(S^{-1}g)_{x_{j}} +    
(W\ast |u|^{2})g.    
$$
Note that we have used the fact that $W/(4 \pi)$ defines, by convolution, a
right inverse for $S$. We analyze each of the three terms. 
\begin{enumerate}
\item
For arbitrary $t$ and $g \in X$:
$$\||u|^{2} S^{-1}g\|_{X} \leq  
\||u|^{2}\|_{L^{1}}   
\|S^{-1}g\|_{L^{\infty}} \leq  
C \sup_{0\leq s \leq T}\|u(\cdotp, s)\|_{X}^{2} \|g\|_{X}.
$$
We have used Sobolev's inequality.
\item  
For arbitrary $t, j$ and $g \in X$:
$$
\|(W_{x_{j}}\ast |u|^{2})(S^{-1}g)_{x_{j}})\|_{X}    
\leq
\|(W_{x_{j}}\ast |u|^{2})\|_{L^{6/5}}    
\|(S^{-1}g)_{x_{j}}\|_{L^{6}}    
\leq
$$
$$
C\|W\|_{L^{6/5}}    
 \sup_{0 \leq s \leq T}\|u(\cdotp, s)\|_{X}^{2} \|g\|_{X}.
$$
We have used the H\"{o}lder, Young, and Sobolev inequalities,
as well as the standard computation of partial derivatives of $W$.
\item
For arbitrary $t$ and $g \in X$:
$$
\|(W \ast |u|^{2}) g\|_{X} \leq  
\|(W \ast |u|^{2})\|_{X} \|g\|_{X}   
\leq \|W\|_{X}
 \sup_{0 \leq s \leq T}\|u(\cdotp, s)\|_{X}^{2} \|g\|_{X}.
$$
We have used Young's inequality and the Schwarz inequality.
\end{enumerate}
This establishes that $B(t)$ is bounded on $X$ for each $t$.
The function space measurability and integrability are discussed
in detail in \cite[Prop.\ 7.1.4]{J1}.
This completes the verification for the final term. 
The verification for the second term is straightforward.
\end{proof}

\end{document}